\documentclass{mmn_sample}

\theoremstyle{plain}
\newtheorem{theorem}{Theorem}
\newtheorem{corollary}{Corollary}
\newtheorem{lemma}{Lemma}

\theoremstyle{definition}
\newtheorem{definition}{Definition}

\theoremstyle{remark}
\newtheorem{remark}{Remark}

\numberwithin{equation}{section}

\allowdisplaybreaks[4]

\begin{document}
\title[Inequalities of Chebyshev-P\'{o}lya-Szeg\"{o} Type]{Inequalities of Chebyshev-P\'{o}lya-Szeg\"{o} Type via Generalized
	Proportional Fractional Integral Operators} 

\author{Saad Ihsan Butt}
\address{COMSATS University Islamabad, Lahore Campus, Pakistan}

\email{saadihsanbutt@gmail.com}

\thanks{The research of the first author has been fully supported by H.E.C. Pakistan under NRPU project 7906.}
\author{Ahmet Ocak Akdem\.{i}r}
\address{Department of Mathematics, Faculty of Science and Letters, A\u{g}ri Ibrahim Cecen University, A\u{g}ri, Turkey}
\email{aocakakdemir@gmail.com}
\author{Alper Ekinci}
\address{Department of Foreign Trade, Bandirma Vocational High School, Bandirma Onyedi Eylul University, Balikesir, Turkey}
\email{alperekinci@hotmail.com}
\author{Muhammad Nadeem}
\address{COMSATS University Islamabad, Lahore Campus, Pakistan}
\email{muhammadnadeem98847@gmail.com}

\begin{abstract}
This study is an example of a solid connection between fractional analysis
and inequality theory, and includes new inequalities of the P\'{o}lya-Szeg%
\"{o}-Chebyshev type obtained with the help of Generalized Proportional
Fractional integral operators. The results have been performed by using
Generalized Proportional Fractional integral operators, some classical
inequalities such as AM-GM inequality, Cauchy-Schwarz inequality and Taylor
series expansion of exponential function. The findings give new approaches
to some types of inequalities that have involving the product of two
functions in inequality theory.
\end{abstract}


\subjclass{26A33, 26D10, 26D15}

\keywords{Chebyshev inequality, Polya-Szeg\"{o} type inequalities, GPF Polya-Szeg\"{o} operator}

\maketitle
\section{Introduction and preliminaries}

Inequalities are a concept that contributes to the solution of many problems
with their applications in different disciplines such as engineering,
physics, statistics and economics as well as being used in many branches of
mathematics. With the help of convex, differentiable, integrable,
continuous, limited, synchronous functions or some other specially defined
functions, many different types of inequality have been proved in the
inequality theory field. These inequalities were brought to the literature
under different names and later became functional with their applications.
We will start by expressing Chebyshev inequality, one of the major
inequality types of inequality theory (see \cite{1}):%
\begin{equation}
T\left( f,g\right) =\frac{1}{b-a}\int_{a}^{b}f\left( x\right) g\left(
x\right) dx-\left( \frac{1}{b-a}\int_{a}^{b}f\left( x\right) dx\right)
\left( \frac{1}{b-a}\int_{a}^{b}g\left( x\right) dx\right)   \label{1}
\end{equation}%
where $f$ and $g$ are two integrable functions which are synchronous on $%
\left[ a,b\right] $, i.e.%
\begin{equation*}
\left( f\left( x\right) -f\left( y\right) \right) \left( g\left( x\right)
-g\left( y\right) \right) \geq 0
\end{equation*}%
for any $x,y\in \left[ a,b\right] $, then the Chebyshev inequality states
that $T\left( f,g\right) \geq 0$.

Chebyshev inequality has been proven for synchronous functions, and has been
the focus of researchers and many different versions have been obtained. We
encourage interested readers to review the following articles. Chebyshev
inequality has been proven for synchronous functions and has been the focus
of researchers and many different versions have been obtained. We encourage
interested readers to review the following articles \cite{2,3,9} and \cite%
{10}.

Another interesting inequality is the P\'olya-Szeg\"o inequality, which
gives boundaries for two functions that can be integrated and their product.
This inequality is given as follows: (see \cite{4})%
\begin{equation*}
\frac{\int_{a}^{b}f^{2}\left( x\right) dx\int_{a}^{b}g^{2}\left( x\right) dx%
}{\left( \int_{a}^{b}f\left( x\right) g\left( x\right) dx\right) ^{2}}\leq 
\frac{1}{4}\left( \sqrt{\frac{MN}{mn}}+\sqrt{\frac{mn}{MN}}\right) ^{2}
\end{equation*}
This inequality is very useful in proving Gr\"uss and Chebyshev type
inequalities. With the help of this inequality, a Chebyshev-Gr\"uss type
inequality is expressed by Dragomir and Diamond as follows in \cite{5}:

\begin{theorem}
	\label{th1} Let $f,g:\left[ a,b\right] \rightarrow  
	\mathbb{R}
	_{+}$ be two integrable functions so that 
	\begin{eqnarray*}
		0 &<&m\leq f\left( x\right) \leq M<\infty \\
		0 &<&n\leq g\left( x\right) \leq N<\infty
	\end{eqnarray*}
	for $x\in \left[ a,b\right] $. Then we have 
	\begin{equation}
	\left\vert T\left( f,g;a,b\right) \right\vert \le\frac{1}{4}\frac{\left(
		M-m\right) \left( N-n\right) }{\sqrt{mnMN}}\left( \frac{1}{b-a}
	\int_{a}^{b}f\left( x\right) dx\right) \left( \frac{1}{b-a}
	\int_{a}^{b}g\left( x\right) dx\right)  \label{1.1}
	\end{equation}
	The constant $\frac{1}{4}$ is best possible in (\ref{1.1}) in the sense it 
	can not be replaced by a smaller constant.
\end{theorem}
\begin{remark} (See \cite{5})
	Assume that the inequality in (\ref{1.1}) holds with a constant $c>0,$ i.e.,
		\begin{equation}
	\left\vert T\left( f,g;a,b\right) \right\vert \le c\frac{\left(
		M-m\right) \left( N-n\right) }{\sqrt{mnMN}}\left( \frac{1}{b-a}
	\int_{a}^{b}f\left( x\right) dx\right) \left( \frac{1}{b-a}
	\int_{a}^{b}g\left( x\right) dx\right) \nonumber
	\end{equation}
	We choose the functions as $f=g$ with
	\begin{equation*}
	f\left( x\right) =\left\{ 
	\begin{array}{c}
	m ,\text{ \ \ \ \ }x\in \left[ a,\frac{a+b}{2}\right] \\ 
	\\ 
	 M ,\text{ \ \ \ \ }x\in [ \frac{a+b}{2},b]%
	\end{array}
	\right.
	\end{equation*}
	where $0<m<M<\infty$, then
		\begin{equation}
	mM \le c(M-m)^2 \label{aaa}
	\end{equation}
	for any  $0<m<M<\infty$. If in \eqref{aaa} we consider $m=1-\epsilon, \ M=1+\epsilon, \ \epsilon \in (0,1)$, then we get $1-(\epsilon)^2 \le 4c$ for any $\epsilon \in (0,1)$, which show that $c \geq \frac{1}{4}$.
\end{remark}
Although fractional analysis origin dates back to the beginning of classical
analysis, it has developed quite rapidly in recent years. Many
mathematicians who researched in this field contributed to this development
and made efforts to strengthen the relationship between fractional analysis
and other fields. With the introduction of new fractional derivative and
integral operators, the application opportunity for many real-world problems
has been revealed. The majority of the new operators came to the fore with
different features such as singularity, location and generalization, and
gained functionality thanks to their effective use in application areas (see
the papers \cite{6,7,51,52,53,mad,mad2,18,19a,mz,25,mz2,farid}). Due to the intensive
work on it, the Riemann-Liouville integral operator is a prominent operator
and is defined as follows.

\begin{definition}
	\label{r-l} Let $f\in L_{1}[a,b]$. The Riemann-Liouville integrals $%
	J_{a+}^{\alpha}f$ and $J_{b-}^{\alpha}f$ of order $\alpha> 0$ with $a\geq0$
	are defined by 
	\begin{equation*}
	J_{a+}^{\alpha}f(t)=\frac{1}{\Gamma (\alpha)}\int_{a}^{t}(t-x)^{%
		\alpha-1}f(x)dx,\,\,\,\,\,\,\,\,\,\,t>a 
	\end{equation*}
	and 
	\begin{equation*}
	J_{b-}^{\alpha}f(t)=\frac{1}{\Gamma (\alpha)}\int_{t}^{b}(x-t)^{%
		\alpha-1}f(x)dx,\,\,\,\,\,\,\,\,\,\,t<b 
	\end{equation*}
	respectively. Here $\Gamma(t)$ is the Gamma function and its definition is $%
	\Gamma(t)=\int_{0}^{\infty}e^{-t}t^{x-1}dx.$ It is to be noted that $%
	J_{a+}^{0}f(t)=J_{b-}^{0}f(t)=f(t)$ in the case of $\alpha=1$, the
	fractional integral reduces to the classical integral.
\end{definition}

We will continue with the Generalized Proportional Fractional integral
operator, which has been described recently and has been the main source of
motivation for many studies in the literature with its use in many areas,
especially the inequality theory. In \cite{jarad}, Jarad et al. identified
the proportional generalized fractional integrals that satisfy many
important features as follows:

\begin{definition}
	\label{gpf} The left and right generalized proportional fracitonal integral
	operators are respectively defined by 
	\begin{equation*}
	_{a+}\mathfrak{J}^{\alpha,\lambda}f(t)=\frac{1}{\lambda^{\alpha}\Gamma(%
		\alpha)} \int_{a}^{t}e^{\left[\frac{\lambda-1}{\lambda}(t-x)\right]}
	(t-x)^{\alpha-1}f(x)dx,\,\,\,\,\,\,\,\,\,\,t>a 
	\end{equation*}
	and 
	\begin{equation*}
	_{b-}\mathfrak{J}^{\alpha, \lambda}f(t)=\frac{1}{\lambda^{\alpha}\Gamma(\alpha)}
	\int_{t}^{b}e^{\left[\frac{\lambda-1}{\lambda}(x-t)\right]}
	(x-t)^{\alpha-1}f(x)dx,\,\,\,\,\,\,\,\,\,\,t<b 
	\end{equation*}
	where $\lambda\in(0,1]$ and $\alpha \in \mathbb{C}$ and $\mathbb{R}(\alpha)>0
	$.
\end{definition}

The main aim of this study is to obtain new P\'{o}lya-Szeg\"{o} type
inequalities by using Generalized Proportional Fractional integral
operators. Taylor series expansion of exponential function is used in
addition to some classical inequalities to obtain main results. The study is
enriched by giving special cases of our results.
\section{Main results}

In this section, we prove certain P\'{o}lya-Szeg\"{o} type integral
inequalities for positive integral functions involving Generalized
Proportional Fractional integral operator.

\begin{lemma}
	\label{2.1.2.1} Assume that f and g are two positive integrable function on $%
	[0,\infty)$. If $v_{1}, v_{2},w_{1}$ and $w_{2}$ are positive functions such
	that 
	\begin{equation}  \label{2.1}
	0<v_{1}(\tau)\leq f(\tau)\leq v_{2}(\tau)
	\end{equation}
	\begin{equation*}
	0<w_{1}(\tau)\leq g(\tau)\leq w_{2}(\tau)
	\end{equation*}
	\noindent for $\tau\in[0,x], \ x>0$, then we have the following inequality;
	\begin{equation}
	\frac{{_{0}^{GP F}}I^{\alpha,p_{1}}[w_{1}w_{2}f^{2}](x){_{0}^{GP F}}%
		I^{\alpha,p_{1}}[v_{1}v_{2}g^{2}](x)}{\Big({_{0}^{GP F}}I^{\alpha,p_{1}}\Big[%
		(v_{1}w_{1}+v_{2}w_{2})fg\Big](x)\Big)^{2}}\leq \frac{1}{4}. \label{bbb}
	\end{equation}
	where $\alpha \in (n,n+1]$ and $n=1,2,3,...$.
\end{lemma}

\begin{proof}
	From \eqref{2.1} for $\tau \in \lbrack 0,x],x>0$, we can write 
	\begin{equation}
	\Bigg(\frac{v_{2}(\tau )}{w_{1}(\tau )}-\frac{f(\tau )}{g(\tau )}\Bigg)\geq 0
	\label{2.3}
	\end{equation}%
	\noindent and 
	\begin{equation}
	\Bigg(\frac{f(\tau )}{g(\tau )}-\frac{v_{1}(\tau )}{w_{2}(\tau )}\Bigg)\geq
	0.  \label{2.4}
	\end{equation}%
	\noindent If we multiply \eqref{2.3} and \eqref{2.4} side by side, we have 
	\begin{equation*}
	\Bigg(\frac{v_{2}(\tau )}{w_{1}(\tau )}-\frac{f(\tau )}{g(\tau )}\Bigg)\Bigg(%
	\frac{f(\tau )}{g(\tau )}-\frac{v_{1}(\tau )}{w_{2}(\tau )}\Bigg)\geq 0.
	\end{equation*}%
	\noindent This implies the following inequality, 
	\begin{equation}
	\Big(v_{1}(\tau )w_{1}(\tau )+v_{2}(\tau )w_{2}(\tau )\Big)f(\tau )g(\tau
	)\geq w_{1}(\tau )w_{2}(\tau )f^{2}(\tau )+v_{1}(\tau )v_{2}(\tau
	)g^{2}(\tau ).  \label{2.5}
	\end{equation}%
	\noindent Since all the functions are positive, $p_{1}\in(0,1], \ x\geq\tau$ and $x>0$, by multiplying both sides of \eqref{2.5} by $\frac{1}{%
		p_{1}^{\alpha }\Gamma (\alpha )}e^{\frac{p_{1}-1}{p_{1}}(x-\tau )}(x-\tau
	)^{\alpha -1}$ and then integrating the resulting inequality with respect to 
	$\tau $ over $(0,x)$, we get 
	\begin{equation*}
	\frac{1}{p_{1}^{\alpha }\Gamma (x)}\int_{0}^{x}e^{\frac{p_{1}-1}{p_{1}}%
		(x-\tau )}(x-\tau )^{\alpha -1}\Big(v_{1}(\tau )w_{1}(\tau )+v_{2}(\tau
	)w_{2}(\tau )\Big)f(\tau )g(\tau )d\tau 
	\end{equation*}%
	\begin{equation*}
	\geq \frac{1}{p_{1}^{\alpha }\Gamma (x)}\int_{0}^{x}e^{\frac{p_{1}-1}{p_{1}}%
		(x-\tau )}(x-\tau )^{\alpha -1}\Big(w_{1}(\tau )w_{2}(\tau )f^{2}(\tau )\Big)%
	d\tau 
	\end{equation*}%
	\begin{equation*}
	+\frac{1}{p_{1}^{\alpha }\Gamma (x)}\int_{0}^{x}e^{\frac{p_{1}-1}{p_{1}}%
		(x-\tau )}(x-\tau )^{\alpha -1}v_{1}(\tau )v_{2}(\tau )g^{2}(\tau )d\tau .
	\end{equation*}%
	Namely, 
	\begin{equation}
	{_{0}^{GPF}}I^{\alpha ,p_{1}}\Big[(v_{1}w_{1}+v_{2}w_{2})fg\Big](x)\geq {%
		_{0}^{GPF}}I^{\alpha ,p_{1}}[w_{1}w_{2}f^{2}](x)+{_{0}^{GPF}}I^{\alpha
		,p_{1}}[v_{1}v_{2}g^{2}](x).
	\end{equation}%
	\noindent Applying the A.M-G.M inequality i.e $(a+b\geq 2\sqrt{ab},a,b\in
	\Re ^{+})$, we have 
	\begin{equation*}
	{_{0}^{GPF}}I^{\alpha ,p_{1}}\Big[(v_{1}w_{1}+v_{2}w_{2})fg\Big](x)\geq 2%
	\sqrt{{_{0}^{GPF}}I^{\alpha ,p_{1}}\Big[w_{1}w_{2}f^{2}\Big](x){_{0}^{GPF}}%
		I^{\alpha ,p_{1}}\Big[v_{1}v_{2}g^{2}\Big](x)}.
	\end{equation*}%
	\noindent This can be written as 
	\begin{equation*}
	{_{0}^{GPF}}I^{\alpha ,p_{1}}\Big[w_{1}w_{2}f^{2}\Big](x){_{0}^{GPF}}%
	I^{\alpha ,p_{1}}\Big[v_{1}v_{2}g^{2}\Big](x)\leq \frac{1}{4}\Big({_{0}^{GPF}%
	}I^{\alpha ,p_{1}}\Big[(v_{1}w_{1}+v_{2}w_{2})fg\Big](x)\Big)^{2}.
	\end{equation*}
\end{proof}

\begin{corollary}
	If we take into account $v_{1}=m,v_{2}=m,w_{1}=n$ and $w_{2}=N$ in \eqref{bbb}, then we
	have the following new inequality;
	\begin{equation*}
	\frac{\Big({_{0}^{GPF}}I^{\alpha,p_{1}}f^{2}\Big)(x)\Big({_{0}^{GPF}}%
		I^{\alpha,p_{1}}g^{2}\Big)(x)}{\Bigg(\Big({_{0}^{GPF}}I^{\alpha,p_{1}}fg\Big)%
		(x)\Bigg)^{2}}\leq\frac{1}{4}\Bigg(\sqrt{\frac{mn}{MN}}+\sqrt{\frac{MN}{mn}}%
	\Bigg)^{2}.
	\end{equation*}
\end{corollary}

\begin{lemma}
	Let f and g be two positive integrable functions on $[0,\infty)$. Assume that
	there exists four positive integrable functions $v_{1},v_{2},w_{1}$ and $%
	w_{2}$ satisfying condition \eqref{2.1}. Then the following inequality
	holds: 
	\begin{equation}  \label{2.7}
	{_{0}^{GPF}}I^{\alpha,p_{1}}[v_{1}v_{2}](x){_{0}^{GPF}}I^{%
		\beta,p_{2}}[w_{1}w_{2}](x)\times{_{0}^{GPF}}I^{\alpha,p_{1}}[f^{2}](x){%
		_{0}^{GPF}}I^{\beta,p_{2}}[g^{2}](x)
	\end{equation}
	\begin{equation*}
	\leq\frac{1}{4}\Bigg({_{0}^{GPF}}I^{\alpha,p_{1}}[v_{1}f](x){_{0}^{GPF}}%
	I^{\beta,p_{2}}[w_{1}g](x)+{_{0}^{GPF}}I^{\alpha,p_{1}}[v_{2}f](x){_{0}^{GPF}%
	}I^{\beta,p_{2}}[w_{2}g](x)\Bigg)^{2}
	\end{equation*}
	where $	\alpha\in(n,n+1]$ and $\beta\in(k,k+1]$, $n,k=0,1,2,3,....$ 
\end{lemma}

\begin{proof}
	From \eqref{2.1}, we get 
	\begin{equation*}
	\Bigg(\frac{v_{2}(\tau)}{w_{1}(\xi)}-\frac{f(\tau)}{g(\xi)}\Bigg)\geq0
	\end{equation*}
	\noindent and 
	\begin{equation*}
	\Bigg(\frac{f(\tau)}{g(\xi)}-\frac{v_{1}(\tau)}{w_{2}(\xi)}\Bigg)\geq0.
	\end{equation*}
	\noindent Which leads to 
	\begin{equation}  \label{2.8}
	\Bigg(\frac{v_{1}(\tau)}{w_{2}(\xi)}+\frac{v_{2}(\tau)}{w_{1}(\xi)}\Bigg)%
	\frac{f(\tau)}{g(\xi)}\geq\frac{f^{2}(\tau)}{g^{2}(\xi)}+\frac{%
		v_{1}(\tau)v_{2}(\tau)}{w_{1}(\xi)w_{2}(\xi)}.
	\end{equation}
	\noindent Multiplying both sides of \eqref{2.8} by $w_{1}(\xi)w_{2}(%
	\xi)g^{2}(\xi)$, we have 
	\begin{equation}  \label{2.9}
	v_{1}(\tau)f(\tau)w_{1}(\xi)g(\xi)+v_{2}(\tau)f(\tau)w_{2}(\xi)g(\xi)\geq
	w_{1}(\xi)w_{2}(\xi)f^{2}(\tau)+v_{1}(\tau)v_{2}(\tau)g^{2}(\xi).
	\end{equation}
	\noindent Multiplying both sides \eqref{2.9} by 
	\begin{equation*}
	\frac{1}{p_{1}^{\alpha}\Gamma(\alpha)}\frac{1}{p_{2}^{\beta}\Gamma(\beta)}
	e^{\frac{p_{1}-1}{p_{1}}(x-\tau)}(x-\tau)^{\alpha-1}e^{\frac{p_{2}-1}{p_{2}}%
		(x-\xi)}(x-\tau)^{\alpha-1}(x-\xi)^{\beta-1}
	\end{equation*}
	\noindent and integrating the resulting inequality with respect to $\tau$
	and $\xi$ over $(0,x)^{2}$, we get 
	\begin{align*}
	\frac{1}{p_{1}^{\alpha}\Gamma(\alpha)}\frac{1}{p_{2}^{\beta}\Gamma(\beta)}%
	\int_{0}^{x}\int_{0}^{x}e^{\frac{p_{1}-1}{p_{1}}(x-\tau)}e^{\frac{p_{2}-1}{%
			p_{2}}(x-\xi)}(x-\tau)^{\alpha-1}(x-\xi)^{\beta-1}v_{1}(\tau)f(\tau)w_{1}(%
	\xi)g(\xi)d\tau d\xi \\
	+\frac{1}{p_{1}^{\alpha}\Gamma(\alpha)}\frac{1}{p_{2}^{\beta}\Gamma(\beta)}%
	\int_{0}^{x}\int_{0}^{x}e^{\frac{p_{1}-1}{p_{1}}(x-\tau)}e^{\frac{p_{2}-1}{%
			p_{2}}(x-\xi)}(x-\tau)^{\alpha-1}(x-\xi)^{\beta-1}v_{2}(\tau)f(\tau)w_{2}(%
	\xi)g(\xi)d\tau d\xi \\
	\geq\frac{1}{p_{1}^{\alpha}\Gamma(\alpha)}\frac{1}{p_{2}^{\beta}\Gamma(\beta)%
	}\int_{0}^{x}\int_{0}^{x}e^{\frac{p_{1}-1}{p_{1}}(x-\tau)}e^{\frac{p_{2}-1}{%
			p_{2}}(x-\xi)}(x-\tau)^{\alpha-1}(x-\xi)^{\beta-1}w_{1}(\xi)w_{2}(\xi)f^{2}(%
	\tau)d\tau d\xi \\
	+\frac{1}{p_{1}^{\alpha}\Gamma(\alpha)}\frac{1}{p_{2}^{\beta}\Gamma(\beta)}%
	\int_{0}^{x}\int_{0}^{x}e^{\frac{p_{1}-1}{p_{1}}(x-\tau)}e^{\frac{p_{2}-1}{%
			p_{2}}(x-\xi)}(x-\tau)^{\alpha-1}(x-\xi)^{\beta-1}v_{1}(\tau)v_{2}(\tau)g^(%
	\xi)d\tau d\xi.
	\end{align*}
	If we re-write the above inequality with the help of the definition of
	operator, we get 
	\begin{align*}
	{_{0}^{GPF}}I^{\alpha,p_{1}}[v_{1}f](x){_{0}^{GPF}}I^{%
		\beta,p_{2}}[w_{1}g](x)+{_{0}^{GPF}}I^{\alpha,p_{1}}[v_{2}f](x){_{0}^{GPF}}%
	I^{\beta,p_{2}}[w_{2}g](x) \\
	\geq {_{0}^{GPF}}I^{\alpha,p_{1}}[f^{2}](x){_{0}^{GPF}}I^{%
		\beta,p_{2}}[w_{1}w_{2}](x)+{_{0}^{GPF}}I^{\alpha,p_{1}}[v_{1}v_{2}](x){%
		_{0}^{GPF}}I^{\beta,p_{2}}[g^{2}](x).
	\end{align*}
	\noindent Applying the A.M-G.M inequality, we have 
	\begin{align*}
	{_{0}^{GPF}}I^{\alpha,p_{1}}[v_{1}f](x){_{0}^{GPF}}I^{%
		\beta,p_{2}}[w_{1}g](x)+{_{0}^{GPF}}I^{\alpha,p_{1}}[v_{2}f](x){_{0}^{GPF}}%
	I^{\beta,p_{2}}[w_{2}g](x) \\
	\geq2\sqrt{{_{0}^{GPF}}I^{\alpha,p_{1}}[f^{2}](x){_{0}^{GPF}}%
		I^{\beta,p_{2}}[w_{1}w_{2}](x)\times{_{0}^{GPF}}I^{%
			\alpha,p_{1}}[v_{1}v_{2}](x){_{0}^{GPF}}I^{\beta,p_{2}}[g^{2}](x)}.
	\end{align*}
	\noindent Which leads to the desired inequality in \eqref{2.7}. \noindent
	The proof is completed.
\end{proof}

\begin{corollary}
	If we set $v_{1}=M$, $v_{2}=M$, $w_{1}=n$ and $w_{2}=N$ in \eqref{2.7}, then we have the following inequality;
	\begin{equation*}
	{_{0}^{GP F}}I^{\alpha,p_{1}}(x){_{0}^{GP F}}I^{\beta,p_{2}}(x)\frac{\Big({%
			_{0}^{GP F}}I^{\alpha,p_{1}}f^{2}\Big)(x)\Big({_{0}^{GP F}}%
		I^{\beta,p_{2}}g^{2}\Big)(x)}{\Bigg(\Big({_{0}^{GP F}}I^{\alpha,p_{1}}f\Big)%
		(x)\Big({_{0}^{GP F}}I^{\beta,p_{2}}g\Big)(x)\Bigg)^{2}}\leq\frac{1}{4}\Bigg(%
	\sqrt{\frac{mn}{MN}}+\sqrt{\frac{MN}{mn}}\Bigg)^{2}
	\end{equation*}
	\noindent Let $a=\frac{p_{1}-1}{p_{1}}$. The Taylor Series of $exp(a(x-\tau))
	$ at the point x is given by 
	\begin{equation*}
	{_{0}^{GP F}}I^{\alpha,p_{1}}(x)=\frac{1}{p_{1}^{\alpha}\Gamma{(\alpha)}}%
	\int_{0}^{x}e^{\frac{p_{1}-1}{p_{1}}(x-\tau)}(x-\tau)^{\alpha-1}d\tau
	\end{equation*}
	\begin{equation*}
	=\frac{1}{p_{1}\Gamma{(\alpha)}}\int_{0}^{x}\sum_{k=0}^{\infty}\frac{\Big(%
		a(x-\tau)\Big)^{k_{1}}}{k_{1}!}(x-\tau)^{\alpha-1}d\tau
	\end{equation*}
	\begin{equation*}
	=\frac{1}{p_{1}^{\alpha}\Gamma{(\alpha)}}\sum_{k=0}^{\infty}\frac{1}{k!}%
	\frac{a^{k_{1}}x^{\alpha+k_{1}}}{\alpha+k_{1}}
	\end{equation*}
	\begin{equation*}
	{_{0}^{GP F}}I^{\beta,p_{2}}(x)=\frac{1}{p_{2}^{\beta}\Gamma{(\beta)}}%
	\sum_{k=0}^{\infty}\frac{b^{k_{2}}x^{\beta+k_{2}}}{k_{2}!(\beta+k_{2})}
	\end{equation*}
	\begin{equation*}
	\frac{1}{p_{1}^{\alpha}\Gamma{(\alpha)}}\frac{1}{p_{2}^{\beta}\Gamma{(\beta)}%
	}\sum_{k_{1}=0}^{\infty}\frac{a^{k_{1}}x^{\alpha+k_{1}}}{(\alpha+k_{1})k!}%
	\sum_{k_{2}=0}^{\infty}\frac{b^{k_{2}}x^{\beta+k_{2}}}{(\beta+k_{2})k_{2}!}%
	\times\frac{\Big({_{0}^{GP F}}I^{\alpha,p_{1}}f^{2}\Big)(x)\Big({_{0}^{GP F}}%
		I^{\beta,p_{2}}g^{2}\Big)(x)}{\Bigg(\Big({_{0}^{GP F}}I^{\alpha,p_{1}}f\Big)%
		(x)\Big({_{0}^{GP F}}I^{\beta,p_{2}}g\Big)(x)\Bigg)^{2}}
	\end{equation*}
	\begin{equation*}
	\leq\frac{1}{4}\Bigg(\sqrt{\frac{mn}{MN}}+\sqrt{\frac{MN}{mn}}\Bigg)^{2}.
	\end{equation*}
\end{corollary}

\begin{lemma}
	Let f and g be two positive integrable function on $[0,\infty)$. Assume that there
	exist four positive integrable functions $v_{1},v_{2},w_{1}$ and $w_{2}$
	satisfying condition \eqref{2.1} then the following inequality holds. 
	\begin{equation}  \label{2.10}
	{_{0}^{GP F}}I^{\alpha,p_{1}}[f^{2}](x){_{0}^{GP F}}I^{%
		\beta,p_{2}}[g^{2}](x)\leq{_{0}^{GP F}}I^{\alpha,p_{1}}[\frac{v_{2}fg}{w_{1}}%
	](x){_{0}^{GP F}}I^{\beta,p_{2}}[\frac{w_{2}fg}{v_{1}}](x).
	\end{equation}
	where $	\alpha\in(n,n+1]$, $\beta\in(k,k+1]$, $n,k=0,1,2,3,...$.
\end{lemma}

\begin{proof}
	Using the condition \eqref{2.1}, we get 
	\begin{equation}  \label{2.11}
	f^{2}(\tau)\leq\frac{v_{2}(\tau)}{w_{1}(\tau)}f(\tau)g(\tau).
	\end{equation}
	\noindent Multiplying both sides of \eqref{2.11} by $\frac{1}{%
		p_{1}^{\alpha}\Gamma{(\alpha)}}e^{\frac{p_{1}-1}{p_{1}}(x-\tau)}(x-\tau)^{%
		\alpha-1}$ and integrating the resulting inequality with respect to $\tau$
	over $(0,x)$, we get 
	\begin{equation*}
	\frac{1}{p_{1}^{\alpha}\Gamma{(\alpha)}}\int_{0}^{x}e^{\frac{p_{1}-1}{p_{1}}%
		(x-\tau)}(x-\tau)^{\alpha-1}f^{2}(\tau)d\tau\leq\frac{1}{p_{1}^{\alpha}\Gamma%
		{(\alpha)}}\int_{0}^{x}e^{\frac{p_{1}-1}{p_{1}}(x-\tau)}(x-\tau)^{\alpha-1}%
	\frac{v_{2}\tau}{w_{1}(\tau)}f(\tau)g(\tau)d\tau
	\end{equation*}
	\begin{equation}  \label{2.12}
	{_{0}^{GP F}}I^{\alpha,p_{1}}[f^{2}](x)\leq{_{0}^{GP F}}I^{\alpha,p}[\frac{%
		v_{2}fg}{w_{1}}](x).
	\end{equation}
	\noindent Similarly, we can write 
	\begin{equation*}
	g^{2}(\xi)\leq\frac{w_{2}(\xi)}{v_{1}(\xi)}f(\xi)g(\xi).
	\end{equation*}
	\noindent By a similar argument, we have 
	\begin{equation*}
	\frac{1}{p_{2}^{\beta}\Gamma{(\beta)}}\int_{0}^{x}e^{\frac{p_{2}-1}{p_{2}}%
		(x-\xi)}(x-\xi)^{\beta-1}g^{2}(\xi)d\xi\leq\frac{1}{p_{2}^{\beta}\Gamma{%
			(\beta)}}\int_{0}^{x}e^{\frac{p_{2}-1}{p_{2}}(x-\xi)}(x-\xi)^{\beta-1}\frac{%
		w_{2}(\xi)}{v_{1}(\xi)}f(\xi)g(\xi)d\xi.
	\end{equation*}
	\noindent Which implies 
	\begin{equation}  \label{2.13}
	{_{0}^{GP F}}I^{\beta,p_{2}}[g^{2}](x)\leq{_{0}^{GP F}}I^{\beta,p_{2}}[\frac{%
		w_{2}fg}{v_{1}}](x).
	\end{equation}
	\noindent Multiplying \eqref{2.12} and \eqref{2.13}, we get then \eqref{2.10}%
	. \noindent Then the desired inequality is obtained such that 
	\begin{equation*}
	{_{0}^{GP F}}I^{\alpha,p_{1}}[f^{2}](x){_{0}^{GP F}}I^{%
		\beta,p_{2}}[g^{2}](x)\leq{_{0}^{GP F}}I^{\alpha,p_{1}}[\frac{v_{2}fg}{w_{1}}%
	](x){_{0}^{GP F}}I^{\beta,p_{2}}[\frac{w_{2}fg}{v_{1}}](x).
	\end{equation*}
\end{proof}

\begin{corollary}
	If we choose $v_{1}=m$, $v_{2}=M$,$w_{1}=n$ and $w_{2}=N$ in \eqref{2.10}, then we have the following inequality; 
	\begin{equation*}
	\frac{\Big({_{0}^{GPF}}I^{\alpha,p_{1}}f^{2}\Big)(x)\Big({_{0}^{GPF}}%
		I^{\alpha,p_{1}}g^{2}\Big)(x)}{\Big({_{0}^{GPF}}I^{\alpha,p_{1}}fg\Big)(x)%
		\Big({_{0}^{GPF}}I^{\alpha,p_{1}}fg\Big)(x)}\leq\frac{MN}{mn}.
	\end{equation*}
\end{corollary}

\begin{theorem}
	Let f and g be two positive integrable function on $[0,\infty)$. Assume that there
	exist four positive integrable functions $v_{1},v_{2},w_{1}$ and $w_{2}$
	satisfying condition \eqref{2.1} then the following inequality holds: 
	\begin{align*}
	\Bigg(\frac{1}{p_{1}^{\alpha}\Gamma{(\alpha)}}\sum_{k_{1}=0}^{\infty}\frac{%
		a^{k_{1}}}{k_{1}!}\frac{x^{\alpha+k_{1}}}{\alpha+k_{1}} \Bigg)\Big({_{0}^{GP
			F}}I^{\beta,p_{2}}fg\Big)(x)+\Bigg(\frac{1}{p_{2}^{\beta}\Gamma{(\beta)}}%
	\sum_{k_{2}=0}^{\infty}\frac{b^{k_{2}}}{k_{2}!}\frac{x^{\beta+k_{2}}}{%
		\beta+k_{2}}\Bigg)\Big({_{0}^{GP F}}I^{\alpha,p_{1}}fg\Big)(x) \\
	-\Big({_{0}^{GP F}}I^{\alpha,p_{1}}f\Big)(x)\Big({_{0}^{GP F}}%
	I^{\beta,p_{2}}g\Big)(x)-\Big({_{0}^{GP F}}I^{\beta,p_{2}}f\Big)(x)\Big({%
		_{0}^{GP F}}I^{\alpha,p_{1}}g\Big)(x)
	\end{align*}
	\begin{equation*}
	\leq\Big|A_{1}(f,v_{1},v_{2})(x)+A_{2}(f,v_{1},v_{2})(x)\Big|^{\frac{1}{2}%
	}\times\Big|A_{1}(g,w_{1},w_{2})(x)+A_{2}(g,w_{1},w_{2})(x)\Big|^{\frac{1}{2}%
	},
	\end{equation*}
	\noindent for $	\alpha\in(n,n+1]$, $\beta\in(k,k+1]$, $n,k=0,1,2,3,...$, where 
	\begin{equation*}
	A_{1}(u,v,w)(x)=\Bigg(\frac{1}{p_{2}^{\beta}\Gamma{(\beta)}}%
	\sum_{k_{2}=0}^{\infty}\frac{b^{k_{2}}}{k_{2}!}\frac{x^{\beta+k_{2}}}{%
		\beta+k_{2}}\Bigg)\frac{\Big({_{0}^{GP F}}I^{\alpha,p_{1}}[(v+w)u](x)\Big)%
		^{2}}{4{_{0}^{GP F}}I^{\alpha,p_{1}}[vw](x)}-\Big({_{0}^{GP F}}%
	I^{\alpha,p_{1}}u\Big)(x)\Big({_{0}^{GP F}}I^{\beta,p_{2}}u\Big)(x)
	\end{equation*}
	and 
	\begin{equation*}
	A_{2}(u,v,w)(x)=\Bigg(\frac{1}{p_{1}^{\alpha}\Gamma{(\alpha)}}%
	\sum_{k_{1}=0}^{\infty}\frac{a^{k_{1}}}{k_{1}!}\frac{x^{\alpha+k_{1}}}{%
		\alpha+k_{1}}\Bigg)\frac{\Big({_{0}^{GP F}}I^{\beta,p_{2}}[(v+w)u](x)\Big)%
		^{2}}{4{_{0}^{GP F}}I^{\beta,p_{2}}[vw](x)}-\Big({_{0}^{GP F}}%
	I^{\alpha,p_{1}}u\Big)(x)\Big({_{0}^{GP F}}I^{\beta,p_{2}}u\Big)(x).
	\end{equation*}
\end{theorem}

\begin{proof}
	Let $f$ and $g$ be two positive integrable functions on $[0,\infty)$ for $%
	\tau,\xi\in(0,x)$ with $x>0$, we define $H(\tau,\xi)$ as 
	\begin{equation*}
	H(\tau,\xi)=\Big(f(\tau)-f(\xi)\Big)\Big(g(\tau)-g(\xi)\Big).
	\end{equation*}
	\noindent Namely 
	\begin{equation}  \label{2.15}
	H(\tau,\xi)=f(\tau)g(\tau)+f(\xi)g(\xi)-f(\tau)g(\xi)-f(\xi)g(\tau).
	\end{equation}
	\noindent Multiplying both sides of \eqref{2.15} by 
	\begin{equation*}
	\frac{1}{p_{1}^{\alpha}\Gamma{(\alpha)}}\frac{1}{p_{2}^{\beta}\Gamma{(\beta)}%
	}(x-\tau)^{\alpha-1}(x-\xi)^{\beta-1}e^{\frac{p_{1}-1}{p_{1}}(x-\tau)} e^{%
		\frac{p_{2}-1}{p_{2}}(x-\xi)}.
	\end{equation*}
	\noindent Then by integrating the resulting inequality with respect to $\tau$
	and $\xi$ over $(0,x)^{2}$, we get 
	\begin{equation*}
	\frac{1}{p_{1}^{\alpha}\Gamma{(\alpha)}}\frac{1}{p_{2}^{\beta}\Gamma{(\beta)}%
	}\int_{0}^{x}\int_{0}^{x}e^{\frac{p_{1}-1}{p_{1}}(x-\tau)} e^{\frac{p_{2}-1}{%
			p_{2}}(x-\xi)}(x-\tau)^{\alpha-1}(x-\xi)^{\beta-1}H(\tau,\xi)d\tau d\xi
	\end{equation*}
	\begin{equation*}
	=\frac{1}{p_{1}^{\alpha}\Gamma{(\alpha)}}\frac{1}{p_{2}^{\beta}\Gamma{(\beta)%
	}}\int_{0}^{x}\int_{0}^{x}e^{\frac{p_{1}-1}{p_{1}}(x-\tau)} e^{\frac{p_{2}-1%
		}{p_{2}}(x-\xi)}(x-\tau)^{\alpha-1}(x-\xi)^{\beta-1}f(\tau)g(\tau)d\tau d\xi
	\end{equation*}
	\begin{equation*}
	+\frac{1}{p_{1}^{\alpha}\Gamma{(\alpha)}}\frac{1}{p_{2}^{\beta}\Gamma{(\beta)%
	}}\int_{0}^{x}\int_{0}^{x}e^{\frac{p_{1}-1}{p_{1}}(x-\tau)} e^{\frac{p_{2}-1%
		}{p_{2}}(x-\xi)}(x-\tau)^{\alpha-1}(x-\xi)^{\beta-1}f(\xi)g(\xi)d\tau d\xi
	\end{equation*}
	\begin{equation*}
	-\frac{1}{p_{1}^{\alpha}\Gamma{(\alpha)}}\frac{1}{p_{2}^{\beta}\Gamma{(\beta)%
	}}\int_{0}^{x}\int_{0}^{x}e^{\frac{p_{1}-1}{p_{1}}(x-\tau)} e^{\frac{p_{2}-1%
		}{p_{2}}(x-\xi)}(x-\tau)^{\alpha-1}(x-\xi)^{\beta-1}f(\tau)g(\xi)d\tau d\xi
	\end{equation*}
	\begin{equation*}
	-\frac{1}{p_{1}^{\alpha}\Gamma{(\alpha)}}\frac{1}{p_{2}^{\beta}\Gamma{(\beta)%
	}}\int_{0}^{x}\int_{0}^{x}e^{\frac{p_{1}-1}{p_{1}}(x-\tau)} e^{\frac{p_{2}-1%
		}{p_{2}}(x-\xi)}(x-\tau)^{\alpha-1}(x-\xi)^{\beta-1}f(\xi)g(\tau)d\tau d\xi
	\end{equation*}
	\begin{equation*}
	\Rightarrow\frac{1}{p_{1}^{\alpha}\Gamma{(\alpha)}}\frac{1}{%
		p_{2}^{\beta}\Gamma{(\beta)}}\int_{0}^{x}\int_{0}^{x}e^{\frac{p_{1}-1}{p_{1}}%
		(x-\tau)} e^{\frac{p_{2}-1}{p_{2}}(x-\xi)}(x-\tau)^{\alpha-1}(x-\xi)^{%
		\beta-1}H(\tau,\xi)d\tau d\xi
	\end{equation*}
	\begin{equation*}
	=\Bigg(\frac{1}{p_{2}^{\beta}\Gamma{(\beta)}}\sum_{k_{2}=0}^{\infty}\frac{%
		b^{k_{2}}}{k_{2}!}\frac{x^{\beta+k_{2}}}{\beta+k_{2}}\Bigg)\Big({_{0}^{GP F}}%
	I^{\alpha,p_{1}}fg\Big)(x)+\Bigg(\frac{1}{p_{1}^{\alpha}\Gamma{(\alpha)}}%
	\sum_{k_{1}=0}^{\infty}\frac{a^{k_{1}}}{k_{1}!}\frac{x^{\alpha+k_{1}}}{%
		\alpha+k_{1}}\Bigg)\Big({_{0}^{GP F}}I^{\beta,p_{2}}fg\Big)(x)
	\end{equation*}
	\begin{equation*}
	-\Big({_{0}^{GP F}}I^{\alpha,p_{1}}f\Big)(x)\Big({_{0}^{GP F}}%
	I^{\beta,p_{2}}g\Big)(x)-\Big({_{0}^{GP F}}I^{\beta,p_{2}}f\Big)(x)\Big({%
		_{0}^{GP F}}I^{\alpha,p_{1}}g\Big)(x).
	\end{equation*}
	\noindent Applying the Cauchy-Schwarz inequality, we
	can write 
	\begin{align*}
	\Bigg|\frac{1}{p_{1}^{\alpha}\Gamma{(\alpha)}}\frac{1}{p_{2}^{\beta}\Gamma{%
			(\beta)}}\int_{0}^{x}\int_{0}^{x}e^{\frac{p_{1}-1}{p_{1}}(x-\tau)} e^{\frac{%
			p_{2}-1}{p_{2}}(x-\xi)}(x-\tau)^{\alpha-1}(x-\xi)^{\beta-1}H(\tau,\xi)d\tau
	d\xi\Bigg| \\
	\leq\Bigg[\frac{1}{p_{1}^{\alpha}\Gamma{(\alpha)}}\frac{1}{%
		p_{2}^{\beta}\Gamma{(\beta)}}\int_{0}^{x}\int_{0}^{x}e^{\frac{p_{1}-1}{p_{1}}%
		(x-\tau)} e^{\frac{p_{2}-1}{p_{2}}(x-\xi)}(x-\tau)^{\alpha-1}(x-\xi)^{%
		\beta-1}f^{2}(\tau)d\tau d\xi \\
	+\frac{1}{p_{1}^{\alpha}\Gamma{(\alpha)}}\frac{1}{p_{2}^{\beta}\Gamma{(\beta)%
	}}\int_{0}^{x}\int_{0}^{x}e^{\frac{p_{1}-1}{p_{1}}(x-\tau)} e^{\frac{p_{2}-1%
		}{p_{2}}(x-\xi)}(x-\tau)^{\alpha-1}(x-\xi)^{\beta-1}f^{2}(\xi)d\tau d\xi \\
	-2\frac{1}{p_{1}^{\alpha}\Gamma{(\alpha)}}\frac{1}{p_{2}^{\beta}\Gamma{%
			(\beta)}}\int_{0}^{x}\int_{0}^{x}e^{\frac{p_{1}-1}{p_{1}}(x-\tau)} e^{\frac{%
			p_{2}-1}{p_{2}}(x-\xi)}(x-\tau)^{\alpha-1}(x-\xi)^{\beta-1}f(\tau)f(\xi)d%
	\tau d\xi\Bigg]^{\frac{1}{2}} \\
	\times\Bigg[\frac{1}{p_{1}^{\alpha}\Gamma{(\alpha)}}\frac{1}{%
		p_{2}^{\beta}\Gamma{(\beta)}}\int_{0}^{x}\int_{0}^{x}e^{\frac{p_{1}-1}{p_{1}}%
		(x-\tau)} e^{\frac{p_{2}-1}{p_{2}}(x-\xi)}(x-\tau)^{\alpha-1}(x-\xi)^{%
		\beta-1}g^{2}(\tau)d\tau d\xi \\
	+\frac{1}{p_{1}^{\alpha}\Gamma{(\alpha)}}\frac{1}{p_{2}^{\beta}\Gamma{(\beta)%
	}}\int_{0}^{x}\int_{0}^{x}e^{\frac{p_{1}-1}{p_{1}}(x-\tau)} e^{\frac{p_{2}-1%
		}{p_{2}}(x-\xi)}(x-\tau)^{\alpha-1}(x-\xi)^{\beta-1}g^{2}(\xi)d\tau d\xi \\
	-2\frac{1}{p_{1}^{\alpha}\Gamma{(\alpha)}}\frac{1}{p_{2}^{\beta}\Gamma{%
			(\beta)}}\int_{0}^{x}\int_{0}^{x}e^{\frac{p_{1}-1}{p_{1}}(x-\tau)} e^{\frac{%
			p_{2}-1}{p_{2}}(x-\xi)}(x-\tau)^{\alpha-1}(x-\xi)^{\beta-1}g(\tau)g(\xi)d%
	\tau d\xi\Bigg]^{\frac{1}{2}}.
	\end{align*}
	\noindent As a consequence 
	\begin{align*}
	\Bigg|\frac{1}{p_{1}^{\alpha}\Gamma{(\alpha)}}\frac{1}{p_{2}^{\beta}\Gamma{%
			(\beta)}}\int_{0}^{x}\int_{0}^{x}e^{\frac{p_{1}-1}{p_{1}}(x-\tau)} e^{\frac{%
			p_{2}-1}{p_{2}}(x-\xi)}(x-\tau)^{\alpha-1}(x-\xi)^{\beta-1}H(\tau,\xi)d\tau
	d\xi\Bigg| \\
	\leq\Bigg[\Bigg(\frac{1}{p_{2}^{\beta}\Gamma{(\beta)}}\sum_{k_{2}=0}^{\infty}%
	\frac{b^{k_{2}}}{k_{2}!}\frac{x^{\beta+k_{2}}}{\beta+k_{2}}\Bigg)\Big({%
		_{0}^{GP F}}I^{\alpha,p_{1}}f^{2}\Big)(x)+\Bigg(\frac{1}{p_{1}^{\alpha}\Gamma%
		{(\alpha)}}\sum_{k_{1}=0}^{\infty}\frac{a^{k_{1}}}{k_{1}!}\frac{%
		x^{\alpha+k_{1}}}{\alpha+k_{1}}\Bigg)\Big({_{0}^{GP F}}I^{\beta,p_{2}}f^{2}%
	\Big)(x) \\
	-2\Big({_{0}^{GP F}}I^{\beta,p_{2}}f\Big)(x)\Big({_{0}^{GP F}}%
	I^{\alpha,p_{1}}f\Big)(x)\Bigg]^{\frac{1}{2}} \\
	\Bigg[\Bigg(\frac{1}{p_{2}^{\beta}\Gamma{(\beta)}}\sum_{k_{2}=0}^{\infty}%
	\frac{b^{k_{2}}}{k_{2}!}\frac{x^{\beta+k_{2}}}{\beta+k_{2}}\Bigg)\Big({%
		_{0}^{GP F}}I^{\alpha,p_{1}}g^{2}\Big)(x)+\Bigg(\frac{1}{p_{1}^{\alpha}\Gamma%
		{(\alpha)}}\sum_{k_{1}=0}^{\infty}\frac{a^{k_{1}}}{k_{1}!}\frac{%
		x^{\alpha+k_{1}}}{\alpha+k_{1}}\Bigg)\Big({_{0}^{GP F}}I^{\beta,p_{2}}g^{2}%
	\Big)(x) \\
	-2\Big({_{0}^{GP F}}I^{\beta,p_{2}}g\Big)(x)\Big({_{0}^{GP F}}%
	I^{\alpha,p_{1}}g\Big)(x)\Bigg]^{\frac{1}{2}}.
	\end{align*}
	\noindent Applying Lemma 1 with $w_{1}(\tau)=w_{2}(\tau)=g(\tau)=1$, we get 
	\begin{align*}
	\Bigg(\frac{1}{p_{2}^{\beta}\Gamma{(\beta)}}\sum_{k_{2}=0}^{\infty}\frac{%
		b^{k_{2}}}{k_{2}!}\frac{x^{\beta+k_{2}}}{\beta+k_{2}}\Bigg)\Big({_{0}^{GP F}}%
	I^{\alpha,p_{1}}f^{2}\Big)(x) \\
	\leq\Bigg(\frac{1}{p_{2}^{\beta}\Gamma{(\beta)}}\sum_{k_{2}=0}^{\infty}\frac{%
		b^{k_{2}}}{k_{2}!}\frac{x^{\beta+k_{2}}}{\beta+k_{2}}\Bigg)\frac{\Big({%
			_{0}^{GP F}}I^{\alpha,p_{1}}[(v_{1}+v_{2})f](x)\Big)^{2}}{4{_{0}^{GP F}}%
		I^{\alpha,p_{1}}[v_{1}v_{2}](x)}.
	\end{align*}
	\noindent This implies that 
	\begin{align*}
	\Bigg(\frac{1}{p_{2}^{\beta}\Gamma{(\beta)}}\sum_{k_{2}=0}^{\infty}\frac{%
		b^{k_{2}}}{k_{2}!}\frac{x^{\beta+k_{2}}}{\beta+k_{2}}\Bigg)\Big({_{0}^{GP F}}%
	I^{\alpha,p_{1}}f^{2}\Big)(x)-\Big({_{0}^{GP F}}I^{\alpha,p_{1}}f\Big)(x)%
	\Big({_{0}^{GP F}}I^{\beta,p_{2}}f\Big)(x) \\
	\leq\Bigg(\frac{1}{p_{2}^{\beta}\Gamma{(\beta)}}\sum_{k_{2}=0}^{\infty}\frac{%
		b^{k_{2}}}{k_{2}!}\frac{x^{\beta+k_{2}}}{\beta+k_{2}}\Bigg)\frac{\Big({%
			_{0}^{GP F}}I^{\alpha,p_{1}}[(v_{1}+v_{2})f](x)\Big)^{2}}{4{_{0}^{GP F}}%
		I^{\alpha,p_{1}}[v_{1}v_{2}](x)}-\Big({_{0}^{GP F}}I^{\alpha,p_{1}}f\Big)(x)%
	\Big({_{0}^{GP F}}I^{\beta,p_{2}}f\Big)(x)
	\end{align*}
	\begin{equation}  \label{2.16}
	=A_{1}(f,v_{1},v_{2})
	\end{equation}
	\noindent and 
	\begin{align*}
	\Bigg(\frac{1}{p_{1}^{\alpha}\Gamma{(\alpha)}}\sum_{k_{1}=0}^{\infty}\frac{%
		a^{k_{1}}}{k_{1}!}\frac{x^{\alpha+k_{1}}}{\alpha+k_{1}}\Bigg)\Big({_{0}^{GP
			F}}I^{\beta,p_{2}}f^{2}\Big)(x)-\Big({_{0}^{GP F}}I^{\alpha,p_{1}}f\Big)(x)%
	\Big({_{0}^{GP F}}I^{\beta,p_{2}}f\Big)(x) \\
	\leq\Bigg(\frac{1}{p_{1}^{\alpha}\Gamma{(\alpha)}}\sum_{k_{1}=0}^{\infty}%
	\frac{a^{k_{1}}}{k_{1}!}\frac{x^{\alpha+k_{1}}}{\alpha+k_{1}}\Bigg)\frac{%
		\Big({_{0}^{GP F}}I^{\beta,p_{2}}[(v_{1}+v_{2})f](x)\Big)^{2}}{4{_{0}^{GP F}}%
		I^{\beta,p_{2}}[v_{1}v_{2}](x)}-\Big({_{0}^{GP F}}I^{\alpha,p_{1}}f\Big)(x)%
	\Big({_{0}^{GP F}}I^{\beta,p_{2}}f\Big)(x)
	\end{align*}
	\begin{equation}  \label{2.17}
	=A_{2}(f,v_{1},v_{2}).
	\end{equation}
	\noindent Similarly, applying Lemma 1 with $v_{1}(\tau)=v_{2}(\tau)=f(\tau)=1
	$, we have 
	\begin{equation*}
	\Bigg(\frac{1}{p_{2}^{\beta}\Gamma{(\beta)}}\sum_{k_{2}=0}^{\infty}\frac{%
		b^{k_{2}}}{k_{2}!}\frac{x^{\beta+k_{2}}}{\beta+k_{2}}\Bigg){_{0}^{GP F}}%
	I^{\alpha,p_{1}}[g^{2}](x)-\Big({_{0}^{GP F}}I^{\alpha,p_{1}}g\Big)(x)\Big({%
		_{0}^{GP F}}I^{\beta,p_{2}}g\Big)(x)
	\end{equation*}
	\begin{equation}  \label{2.18}
	\leq A_{1}(g,w_{1},w_{2})
	\end{equation}
	\noindent and 
	\begin{equation*}
	\Bigg(\frac{1}{p_{1}^{\alpha}\Gamma{(\alpha)}}\sum_{k_{1}=0}^{\infty}\frac{%
		a^{k_{1}}}{k_{1}!}\frac{x^{\alpha+k_{1}}}{\alpha+k_{1}}\Bigg){_{0}^{GP F}}%
	I^{\beta,p_{2}}[g^{2}](x)-\Big({_{0}^{GP F}}I^{\alpha,p_{1}}g\Big)(x)\Big({%
		_{0}^{GP F}}I^{\beta,p_{2}}g\Big)(x)
	\end{equation*}
	\begin{equation}  \label{2.19}
	\leq A_{2}(g,w_{1},w_{2}).
	\end{equation}
	\noindent Using \eqref{2.16}-\eqref{2.19}, we conclude the result.
\end{proof}

\begin{theorem}
	Let f and g be two positive integrable function on $[0,\infty)$. Assume that there
	exist four positive integrable functions $v_{1},v_{2},w_{1}$ and $w_{2}$
	satisfying condition \eqref{2.1} then the following inequality holds: 
	\begin{equation}  \label{2.20}
	\Bigg|\Bigg(\frac{1}{p_{1}^{\alpha}\Gamma{(\alpha)}}\sum_{k_{1}=0}^{\infty}%
	\frac{a^{k_{1}}}{k_{1}!}\frac{x^{\alpha+k_{1}}}{\alpha+k_{1}}\Bigg){_{0}^{GP
			F}}I^{\alpha,p_{1}}[fg](x)-\Big({_{0}^{GP F}}I^{\alpha,p_{1}}f\Big)(x)({%
		_{0}^{GP F}}I^{\alpha,p_{1}}g)(x)\Bigg|
	\end{equation}
	\begin{equation*}
	\leq\Bigg|A(f,v_{1},v_{2})(x)A(g,w_{1},w_{2})(x)\Bigg|^{\frac{1}{2}}
	\end{equation*}
	\noindent for $	\alpha\in(n,n+1]$, $\beta\in(k,k+1]$, $n,k=0,1,2,3,...$, where 
	\begin{equation*}
	A(u,v,w)(x)=\Bigg(\frac{1}{p_{1}^{\alpha}\Gamma{(\alpha)}}%
	\sum_{k_{1}=0}^{\infty}\frac{a^{k_{1}}}{k_{1}!}\frac{x^{\alpha+k_{1}}}{%
		\alpha+k_{1}}\Bigg)\frac{\Bigg({_{0}^{GP F}}I^{\alpha,p_{1}}[(v+w)u](x)\Bigg)%
		^{2}}{4{_{0}^{GP F}}I^{\alpha,p_{1}}[vw](x)}-\Bigg(\Big({_{0}^{GP F}}%
	I^{\alpha,p_{1}}u\Big)(x)\Bigg)^{2}.
	\end{equation*}
\end{theorem}

\begin{proof}
	Setting $\alpha=\theta$ in Theorem 2, we obtain \eqref{2.20}.
\end{proof}

\begin{corollary}
Assume that all the assumptions of Theorem 3 satify, then we have the following inequality;
	\begin{align*}
	\Bigg|\Bigg(\frac{1}{p_{1}^{\alpha}\Gamma{(\alpha)}}\sum_{k_{1}=0}^{\infty}%
	\frac{a^{k_{1}}}{k_{1}!}\frac{x^{\alpha+k_{1}}}{\alpha+k_{1}}\Bigg){_{0}^{GP
			F}}I^{\alpha,p_{1}}[fg](x)-\Big({_{0}^{GP F}}I^{\alpha,p_{1}}f\Big)(x)\Big({%
		_{0}^{GP F}}I^{\alpha,p_{1}}g\Big)(x)\Bigg| \\
	\leq\Bigg(\frac{1}{p_{1}^{\alpha}\Gamma{(\alpha)}}\sum_{k_{1}=0}^{\infty}%
	\frac{a^{k_{1}}}{k_{1}!}\frac{x^{\alpha+k_{1}}}{\alpha+k_{1}}\Bigg)\frac{(M-m)(N-n)}{4\sqrt{MmNn}}\times\Big({_{0}^{GP F}}I^{\alpha,p_{1}}f%
	\Big)(x)\Big({_{0}^{GP F}}I^{\alpha,p_{1}}g\Big)(x).
	\end{align*}
\end{corollary}
\begin{proof}
If we set $v_{1}=m$, $v_{2}=M$, $w_{1}=n$ and $w_{2}=N$ in \eqref{2.20}, then the proof is completed. We omit the details.
\end{proof}
\section{Conflict of interest}

All authors declare no conflicts of interest in this paper.
\section{Acknowledgement}

The research of the first author has been fully supported by H.E.C. Pakistan
under NRPU project 7906.
 
\bibliographystyle{\mmnbibstyle}
\bibliography{\jobname}

\end{document}